\documentclass[11pt]{article}

\usepackage{amsmath,amsthm,amssymb,url}
\usepackage{tikz}
\usetikzlibrary{cd}
\usetikzlibrary{patterns}
\usepackage{tkz-graph}
\usepackage{authblk}
\usepackage{xcolor}
\usepackage{longtable}
\usepackage{xfrac} 
\usepackage{relsize}
\usepackage{nccmath}
\usepackage{bm}
\usepackage{enumitem}
\usepackage{etoolbox}
\makeatletter
\patchcmd\longtable{\par}{\if@noskipsec\mbox{}\fi\par}{}{}
\makeatother
\usepackage{caption}

\usepackage[top=1.25in, bottom=1.25in, left=1.25in, right=1.25in]{geometry}

\newcommand*{\zed}{{\ensuremath{\mathbb{Z}}}}
\newcommand*{\eff}{{\ensuremath{\mathbb{F}}}} 
\newcommand*{\Aut}{\ensuremath{\textsc{aut}}}
\newcommand*{\GCD}{\ensuremath{\textsc{gcd}}}
\newcommand{\RAR}[2]{\rule[-.1in]{0in}{.4in}\xrightarrow{\makebox[#1][c]{$\textstyle #2 $}}}

\theoremstyle{definition}
 \newtheorem{construction}{Construction}[section]
 \newtheorem{definition}[construction]{Definition}
 \newtheorem{remark}[construction]{Remark}
 \newtheorem{example}[construction]{Example}

 \newtheorem*{construction*}{Construction}
 \newtheorem*{definition*}{Definition}
 \newtheorem*{remark*}{Remark}
 \newtheorem*{example*}{Example}
 \newtheorem*{exercise*}{Exercise}
 \newtheorem*{observation*}{Observation}

\theoremstyle{plain}
 \newtheorem{theorem}[construction]{Theorem}
 
 \newtheorem{lemma}[construction]{Lemma}

 \newtheorem{corollary}[construction]{Corollary}

 \newtheorem*{fact*}{Fact}
 \newtheorem*{theorem*}{Theorem}
 \newtheorem*{lemma*}{Lemma}
 \newtheorem*{claim*}{Claim}
 \newtheorem*{proposition*}{Proposition}
 \newtheorem*{corollary*}{Corollary}
 \newtheorem*{question*}{Question}

\numberwithin{figure}{section}
\numberwithin{table}{section}

\title{Near-factorizations of dihedral groups}
\author[1]{Donald L.\ Kreher}
\affil[1]{Department of Mathematical Sciences\\
Michigan Technological University\
Houghton, MI 49931-1295, U.S.A.}
\author[2]{Maura B.\ Paterson}
\affil[2]{School of Computing and Mathematical Sciences, Birkbeck, University of London, Malet St, London WC1E 7HX, UK}
\author[3]{Douglas R.\ Stinson\thanks{D.R.\ Stinson's research is supported by  NSERC discovery grant RGPIN-03882.}}
\affil[3]{David R.\ Cheriton School of Computer Science\\University of Waterloo\\ Waterloo ON, N2L 3G1\\Canada}
\begin{document}
\maketitle

\begin{abstract}
We investigate near-factorizations of nonabelian groups, concentrating on dihedral groups.
We show that some known constructions of near-factorizations in dihedral groups yield equivalent near-factorizations. In fact, there are very few known examples of nonequivalent near-factorizations in dihedral or other nonabelian groups; we provide some new examples with the aid of the computer. We also analyse a construction for near-factorizations in dihedral groups from near-factorizations in cyclic groups, due to P\^{e}cher, and we investigate when nonequivalent near-factorizations can be obtained by this method. 
\end{abstract}

\section{Introduction}
\label{intro.sec}

Let $(G,\cdot)$ be a finite multiplicative group with identity $e$. For $A, B \subseteq G$, define $AB = \{ gh: g \in A, h \in B\}$. We say that $(A,B)$ is a \emph{near-factorization} of $G$
if $|A| \times |B| = |G|-1$ and $G-e =AB$. In the case where we have an additive group $(G,+)$ with identity $0$, the second condition becomes $G-0 =A+B$. Further, $(A,B)$ 
is a \emph{$(k,\ell)$-near-factorization} of $G$ if $|A|=k$ and $|B| =\ell$, which requires 
$k\ell = |G| -1$. 
There is always a \emph{trivial} $(1,|G|-1)$-near-factorization of $G$ given by $A = \{e\}$, $B = G - e$. 
A near-factorization with $|A| > 1$ and $|B| > 1$ is \emph{nontrivial}.

Near-factorizations apparently first appeared in an article by de Bruijn \cite{deB} in 1956, under the name ``degenerated British number systems.'' The term ``near-factorization'' first appears in the 1990 paper
by de Caen {\it et al.} \cite{CGHK}. The paper  \cite{CGHK} includes several interesting results, including the first construction for near-factorizations of dihedral groups.
Other papers on near-factorizations relevant to our work include \cite{BHS} and \cite{Pech}.

\begin{example}
\label{zed15.ex} A $(3,5)$-near-factorization of $\zed_{16}$ is given by
$A = \{ 0,1,15\}$ and $B = \{ 2,5,8,11,14\}$. 
\end{example}

A subset $X$ of a multiplicative group $(G,\cdot)$ is \emph{symmetric} if $X = X^{-1}$, where
$X^{-1} = \{ x^{-1} : x \in X\}$. A subset $X$ of an additive group $(G,+)$ is \emph{symmetric} if $X = -X$, where
$-X = \{ -x : x \in X\}$. A near-factorization $(A,B)$ is \emph{symmetric} if $A$ and $B$ are both 
symmetric. We observe that the near-factorization given in Example \ref{zed15.ex} is symmetric.

We now define the notion of equivalence of near-factorizations.
Suppose $(A,B)$ is a near-factorization of a multiplicative group $(G, \cdot)$, $f \in \Aut(G)$ and $h \in G$. Then
\[
(f(A)h)\,(h^{-1}f(B)) =  f(A)f(B)=f(AB)=f(G-e)=f(G)-f(e)=G-e.
\]
Therefore $(f(A)h, h^{-1}f(B))$ is also a near-factorization of $G$ and we say that
these two
near-factorizations  
of $G$ are \emph{equivalent}.
That is, two near-factorizations $(A,B)$ and $(A',B')$ of $G$ are {equivalent} if
\[
\big(A',B'\big) = \big(f(A)h,h^{-1}f(B)\big)
\]
for some $f \in \Aut(G)$ and $h \in G$.
We also define the \emph{canonical form} of a near-factorization $(A,B)$ to be the minimum equivalent representation of $(A,B)$ under the lexicographic ordering.

In the case of an additive group $(G,+)$, two near-factorizations $(A,B)$ and $(A',B')$ of $G$ are {equivalent} if
\[
\big(A',B'\big) = \big(f(A)+h,(-h) + f(B)\big).
\]

The automorphism group of $(\zed_n,+)$ consists of all mappings
$x \mapsto ax$, where $a \in \zed_n^*$ (as usual, $\zed_n^*$ denotes the group of units modulo $n$). Therefore, if $(A,B)$ and $(A',B')$ are equivalent near-factorizations of
$(\zed_n,+)$, then $A' = aA + h$ and $B' = aB - h$, where $a \in \zed_n^*$, $h \in \zed_n$. 

\begin{example}
The $(3,5)$-near-factorization of $\zed_{16}$  given by
$A = \{ 0,1,15\}$ and $B = \{ 2,5,8,11,14\}$ is equivalent to the 
near-factorization $A' = 7A + 2 = \{ 2,9,11\}$ and $B' = 7B - 2 = \{ 0,1,6,11,12\}$.
\end{example} 

We define the \emph{equivalence map} $\Phi_{f,h}$ on the set of near-factorizations of a multiplicative group, say 
$\{ (A,B) : G\setminus\{e\}=AB \}$, by
\[
\Phi_{f,h}: (A,B)\mapsto \big(f(A)h,h^{-1}f(B)\big).
\]
The following lemma is straightforward.
\begin{lemma}\label{Phi Inverse}
$\Phi_{f,h}^{-1} = \Phi_{f^{-1}, f^{-1}(h^{-1}) }.$
\end{lemma}
\begin{proof}
Denote
\begin{align*}
A' &=  f(A)h\\
B' &= h^{-1}f(B).
\end{align*}
Then
\begin{align*}
A &= f^{-1}(A'h^{-1})\\
&= f^{-1}(A) f^{-1}(h^{-1})\\
B &= f^{-1}(hB')\\
&= f^{-1}(h) f^{-1}(B')\\
&= (f^{-1}(h^{-1}))^{-1} f^{-1}(B'),
\end{align*}
since $(f^{-1}(h^{-1}))^{-1} = f^{-1}(h)$.
The result follows.
\end{proof}

The following theorem is useful. We state it in the setting of additive groups.

\begin{theorem}
\cite[Proposition 2]{CGHK}
\label{translation.thm}
Suppose $(A,B)$ is a {near-factorization} of the abelian group $(G,+)$. Then there is an element $g \in G$ such that $(A+g,B-g)$ is a symmetric near-factorization of  $(G,+)$. 
\end{theorem}


\medskip

In this paper, we  study near-factorizations of nonabelian groups, concentrating on dihedral groups. In the rest of Section \ref{intro.sec}, we review some connections between generalized strong external difference families and near-factorizations, we summarize some basic properties of dihedral groups and near-factorizations in dihedral groups.
In Section \ref{32.sec}, we revisit two known constructions of near-factorizations of dihedral groups and show that they yield equivalent near-factorizations. This contradicts statements made in \cite{BHS}. In Section \ref{pech.sec}, we study a method due to P\^{e}cher that constructs a strongly symmetric near-factorization of the dihedral group $D_n$ (for the definition of ``strongly symmetric,'' see Section \ref{properties.sec}) from a symmetric near-factorization of the cyclic group $\zed_{2n}$ when $n$ is odd. We show that equivalent near-factorizations are mapped to equivalent near-factorizations by this construction. We also prove a converse assertion, provided that a certain numerical condition is satisfied.
In Section \ref{nonabelian.sec}, we show the existence of nonequivalent near-factorizations in two dihedral groups and two nonabelian groups that are not dihedral groups.

\subsection{Near-factorizations and SEDFs}

Near-factorizations are closely related to \emph{strong external difference families} and \emph{generalized strong external difference families}, which have been the topic of several recent papers. We recall two standard definitions in the setting of an additive group.
For two disjoint subsets $A,B$ of an additive  group $G$, define the multiset $\mathcal{D}(B,A)$ as follows:
\[ \mathcal{D}(B,A) = \{ y - x: y \in B, x \in A \}.  \]

\begin{definition}[Strong external difference family \cite{PS16}]
Let $G$ be an additive  group of order $n$. Suppose $m \geq 2$. 
An {\em $(n, m, \ell; \lambda)$-strong external difference family} (or {\em $(n, m, \ell; \lambda)$-SEDF}) is a set of $m$ disjoint $\ell$-subsets of $G$, say $\mathcal{A} = (A_0,\dots,A_{m-1})$, such that  the following multiset equation holds for each $i\in\{0,1,\dotsc,m-1\}$:
\[
\bigcup_{j\neq i} \mathcal{D}(A_i, A_j) = \lambda (G \setminus \{0\}).
\]
\end{definition}

\begin{definition}[Generalized strong external difference family \cite{PS16}]
Let $G$ be an additive abelian group of order $n$.
An \textit{$(n,m;\ell_1, \dots , \ell_m;\lambda_1, \dots , \lambda_m)$-generalized strong external difference family} 
(or \textit{$(n,m;\ell_1, \dots , \ell_m;\lambda_1, \dots , \lambda_m)$-GSEDF})
is a set of $m$ disjoint subsets of $G$, say $A_1, \dots , A_m$, such that
$|A_i| = \ell_i$ for $1 \leq i \leq m$ and 
the following multiset equation holds for every $i$, $1 \leq i \leq m$:
\[ \bigcup _{\{ j : j \neq i\} }^{} \mathcal{D}(A_i,A_j)   = \lambda_i (G \setminus \{0\}).\]
\end{definition}

It is obvious that an $(n,m,\ell,\lambda)$-SEDF is an 
$(n,m;\ell, \dots , \ell;\lambda, \dots , \lambda)$-GSEDF.
Also, in a GSEDF having $m=2$ sets, it must be the case that $\lambda_1 = \lambda_2$.

The following  simple lemma relates GSEDFs to near-factorizations of groups.

\begin{lemma}
\label{GSEDF-NF}
\cite[Lemma 3.7]{PS24}
Let $(G,+)$ be an additive group with $|G| = n$. Then there exists a $(k,\ell)$-near-factorization of $G$ if and only if
there exists an $(n,2;k,\ell;1,1)$-GSEDF in $G$.
\end{lemma}
\begin{proof}
Suppose $(A,B)$ is a $(k,\ell)$-near-factorization of $G$. Then it is straightforward to see that 
$-A$ and $B$ comprise an $(n,2;k,\ell;1,1)$-GSEDF in $G$. Conversely, if $A_1$ and $A_2$ comprise an $(n,2;k,\ell;1,1)$-GSEDF in $G$, then $(-A_1,A_2)$ is a $(k,\ell)$-near-factorization of $G$.
\end{proof}

Since they were first defined in 2016 in \cite{PS16}, there have been numerous papers that have studied SEDFs. 
There are various nonexistence results for SEDFs. In fact, the only parameter sets for which  $(n,m,\ell,\lambda)$-SEDFs are known to exist are the following:
\begin{itemize}
\item $(n,n,1,1)$-SEDFs (these are ``trivial'' SEDFs, consisting of all possible singleton sets)
\item $(n,m,2,\lambda)$-SEDFs (several constructions for SEDFs consisting of two sets are known; see \cite{BJWZ,HJN,HP})
\item a $(243,11,22,20)$-SEDF exists in $\eff_{243}$ (\cite{JeLi,WYFF})
\end{itemize}
We should also mention the paper \cite{LNC}, which specifically studies GSEDFs.


In this paper, we mainly use the language of near-factorizations. However, in view of Lemma \ref{GSEDF-NF}, every result on near-factorizations is automatically a result on GSEDFs having two sets.


\subsection{Dihedral groups}
\label{notation.sec}

We begin by summarizing some relevant material concerning dihedral groups.
The dihedral group $D_{n}$ of order $2n$, $n>2$ has the presentation
\begin{equation}
\label{dihedral.eq}
D_{n} = 
\left\langle
a,b : a^2=b^n= abab = e
\right\rangle,
\end{equation}
where $e$ is the identity element. $D_{n}$ is the automorphism group
of the $n$-gon and can be seen as a set of $n$ \emph{rotations} $\{b^j : 0\leq j < n\}$ and $n$ \emph{reflections} $\{ab^j : 0\leq j \leq n-1\}$.
Thus
\[
D_{n} = 
\left\{
a^i b^j : 0\leq i < 2, 0\leq j \leq n-1
\right\}.
\]
Dihedral group multiplication rules are as follows:
\begin{align*} (ab^i)(ab^j) &= b^{j-i}\\
b^i(ab^j) &= ab^{j-i}\\
(b^i)(b^j) &= b^{i+j}\\
(ab^i)(b^j) &= ab^{i+j},
\end{align*}
where exponents of $b$ are evaluated modulo $n$.
It is easily verified that  $(ab^i)^{-1} = ab^i$ and $(b^i)^{-1} = b^{-i} = b^{n-i}$.

All reflections have order two; the order of a rotation $b^j$ is $n / \GCD(n,j)$.

\medskip

Let $\Aut(D_{n})$ denote the automorphism group of $D_n$. If $f \in \Aut(D_{n})$, then $f$ is completely determined by
$x = f(a)$ and $y=f(b)$. Indeed $x$ and $y$
can be any pair of elements
such that $x$ is a reflection and $y$ is a rotation of order $n$ of the $n$-gon.
\begin{itemize}
\item
There are $\varphi(n)$ options for $y$, namely $y=b^i$, where $i \in \zed_n^*$. 
\item
There are $n$ options for $x$, namely $x=ab^j$, where $j \in \zed_n$.
\end{itemize}
For $i,j$, where $0\leq i \leq n-1$, $\GCD(i,n)=1$ and $0\leq j \leq n-1$,
we define $f=f_{i,j} $ by $f(b)=b^i$, $f(a)=ab^j$.
Because $f$ is an automorphism, we have 
\begin{align*}
f(a^e b^h) 
&= f(a^e)f(b^h)\\
&= f(a)^e f(b)^h\\
&= (a b^j)^e (b^i)^h \\
&=
\left\{\begin{array}{ll}
b^{ih}   &\text{if $e=0$}\\
ab^{j+ih}&\text{if $e=1$}.
\end{array}\right.
\end{align*}
Hence, the automorphism group of $D_{n}$ is 
\[
\Aut(D_{n}) = \{ f_{i,j}:
\text{$\GCD(i,n)=1$ and $0\leq j \leq n-1$}\},
\]
where the automorphism $f_{i,j}$ is defined as follows:
\[
f_{i,j}(a^e b^h) = a^eb^{je+ih},
\] 
for  $e = 0,1$ and $0\leq h \leq n-1$.

Suppose we form the composition $f_{i,j} \circ f_{i',j'}$ of two automorphisms $f_{i,j}$ and $f_{i',j'}$.
We have 
\begin{align*}
f_{i',j'}(f_{i,j}(a^e b^h)) 
&= f_{i',j'}((a b^j)^e (b^i)^h )\\
&= f_{i',j'}((a b^j)^{e}) f_{i',j'}((b^i)^{h})\\
&= f_{i',j'}((a b^j)^{e}) (b^h)^{ii'}\\
&=
\left\{\begin{array}{ll}
b^{hii'}   &\text{if $e=0$}\\
ab^{j'+ji' + hii'}&\text{if $e=1$}.
\end{array}\right.\end{align*}
Hence we have
\begin{equation}
\label{composition.eq}
f_{i,j} \circ f_{i',j'} = f_{ii',j + j'i}.
\end{equation}
\medskip

The identity automorphism is $f_{1,0}$. From this, it is easy to see that 
the inverse of $f_{i,j}$ can immediately be determined from 
(\ref{composition.eq}) to be 
\begin{equation}
\label{inverse.eq}
f_{i,j}^{-1}=f_{i^{-1},-i^{-1}j}. 
\end{equation}
Note that subscripts are evaluated modulo $n$ in (\ref{composition.eq}) and (\ref{inverse.eq}).

\subsection{Properties of near-factorizations of dihedral groups}
\label{properties.sec}

Because $|D_n| = 2n$ is even, it follows immediately that $k$ and $\ell$ are both odd
if a $(k,\ell)$-near-factorization of $D_n$ exists.

\begin{example}
\label{E1.exam}
Define $A=\{e,b,a\}$ and $B=\{b^2,b^5,ab,ab^4,ab^7\}$. It is easy to verify that 
$(A,B)$ is a $(3,5)$-near-factorization of $D_8$. We have
\begin{align*}
eB &= \{b^2,b^5,ab,ab^4,ab^7\}\\
bB &= \{b^3,b^6,a,ab^3,ab^6\}\\
aB &= \{ab^2,ab^5,b,b^4,b^7\}.
\end{align*}
The union of these three sets is $D_8 - e$.
$\hfill\blacksquare$
\end{example}

In Example \ref{E1.exam}, we see that $A$ contains one reflection and two rotations, and
$B$ contains two rotations and three reflections. In general, we have the following result.

\begin{theorem}
\label{size.thm}
If $(A,B)$ is  a $(k,\ell)$-near-factorization of $D_n$,
then  the number of rotations in  $A$ and $B$ is respectively
\begin{enumerate}[label={\textup{(\roman*)}}]
\item  $\tfrac{k-1}{2}$ and $\tfrac{\ell+1}{2}$; 
or \label{B}
\item  $\tfrac{k+1}{2}$ and $\tfrac{\ell-1}{2}$.
\label{A}
\end{enumerate}
\end{theorem}

\begin{proof}
$D_n$ is a group of order $2n$. The $n$ rotations in $D_n$ form a normal subgroup.
Let $\eta:D_n \mapsto \zed_2$ be the canonical homomorphism 
whose kernel is the subgroup of rotations. 
Then
$x=|\eta^{-1}(0)\cap A|$ and $y=|\eta^{-1}(0)\cap B|$ are the number of rotations in $A$ and $B$ respectively. Then
the number of reflections in $D_n$ is
\[
n= x(\ell-y)+(k-x)y\]
and the number of non-identity rotations in $D_n$ is
\[n-1= xy+(k-x)(\ell-y).\]
Subtracting, we obtain
\begin{align*}
1&= x(\ell-y)+(k-x)y - xy-(k-x)(\ell-y)\\
 &= (2x-k)(\ell-2y).
\end{align*}
Hence, either $2x-k=1$ and $\ell-2y=1$, which yields option~\ref{A},
or 
 $2x-k=-1$ and $\ell-2y=-1$, which yields option~\ref{B}.
\end{proof}

\begin{remark*}
Theorem \ref{size.thm} is a special case of a much more general result proven by
P\^{e}cher in \cite[Theorem 1]{Pech}. The proof of \cite[Theorem 1]{Pech} is quite lengthy, so we have included a proof of Theorem \ref{size.thm} because it is so simple.
\end{remark*}

We now present an example illustrating equivalent near-factorizations in $D_n$.

\begin{example} Recall the near-factorization of $D_8$ given in Example \ref{E1.exam}:
$A=\{e,b,a\}$ and $B=\{b^2,b^5,ab,ab^4,ab^7\}$. Let $f = f_{3,2}$ and $h = ab$.
We compute
\begin{align*}
f(A)h &= \{f_{3,2}(e), f_{3,2}(b), f_{3,2}(a)\} (ab)\\
&= \{e, b^3, ab^2\} (ab)\\
&= \{ ab, ab^6, b^7\} 
\end{align*}
and
\begin{align*}
h^{-1}f(B) &= (ab) \{f_{3,2}(b^2), f_{3,2}(b^5), f_{3,2}(ab), f_{3,2}(ab^4), f_{3,2}(ab^7)\}\\
&= ab \{b^6, b^7, ab^5, ab^6, ab^7\}\\
&= \{ab^7, a, b^4, b^5, b^6 \}.
\end{align*}
To verify that this is a near-factorization, we compute
\begin{align*}
ab h^{-1}f(B) &= ab \{ab^7, a, b^4, b^5, b^6 \}\\
&= \{b^6, b^7, ab^5, ab^6, ab^7 \}\\
ab^6 h^{-1}f(B) &= ab^6 \{ab^7, a, b^4, b^5, b^6 \}\\
&=  \{b, b^2, ab^2, ab^3, ab^4 \}\\
b^7 f(B) &= b^7 \{ab^7, a, b^4, b^5, b^6 \}\\
& = \{a, ab, b^3, b^4, b^5 \}.
\end{align*}
$\hfill\blacksquare$
\end{example}

Recall that a subset $X$ of a multiplicative group $(G,\cdot)$ is \emph{symmetric} if $X = X^{-1}$, where
$X^{-1} = \{ x^{-1} : x \in X\}$, and a near-factorization $(A,B)$ is \emph{symmetric} if $A$ and $B$ are both 
symmetric. A subset $X$ of the dihedral group $D_n$ is \emph{strongly symmetric} if 
$x = a^ib^j \in X$ if and only if $a^ib^{-j} \in X$, for  $i = 0,1$ and for all $j$, $0 \leq j \leq n-1$. 
Note that a strongly symmetric subset of $D_n$ is also a symmetric subset of $D_n$. A near-factorization 
$(A,B)$ of $D_n$ is \emph{strongly symmetric} if $A$ and $B$ are both 
strongly symmetric.


\section{Constructions for near-factorizations of dihedral groups}
\label{32.sec}

It was shown by de Caen {\it et al.}\ in 1990 in \cite[Example 2]{CGHK}
that there is a $(k,\ell)$-near-factorization of $D_n$ for all pairs of positive integers $(k,\ell)$ such that $k \ell = 2n-1$.
Later, in 2008, Basc\'{o} {\it et al.}\ proved in \cite[Corollary 1]{BHS} that there is a $(2^r-1,2^r+1)$-near-factorization of $D_n$ when $n = 2^{2r-1}$. We will show that the Basc\'{o} {\it et al.}\ near-factorizations
are equivalent to the de Caen {\it et al.}\ near-factorizations when $k = 2^r-1$ and $n = 2^{2r-1}$.

We begin with an example, to show that these two constructions of $(7,9)$-near-factorizations of $D_{32}$ are equivalent.

\begin{example}
\label{32.exam}
The  de Caen {\it et al.} near-factorization of $D_{32}$  (see
\cite[example 2, page 55]{CGHK}) is
\begin{align*}
A&=\{b,b^2,b^3,a,ab,ab^2,ab^3\}\\
B&=\{e,b^7,b^{14},b^{21},b^{28},ab^7,ab^{14},ab^{21},ab^{28}\}.
\end{align*}
Choosing $h_1=ab^3$ and $f_1=f_{1,0}$ we have:
\begin{align*}
f_1(A)h_1&=\{e,b,b^2,b^3,a,ab,ab^2\}\\
h_1^{-1}f_1(B)&=\{b^4,b^{11},b^{18},b^{25},ab^3,ab^{10},ab^{17},ab^{24},ab^{31}\}.
\end{align*}
This is in fact the canonical form of $(A,B)$.
 
The Basc\'{o} {\it et al.} near-factorization of $D_{32}$  (see
\cite[page 501]{BHS}) is
\begin{align*}
A'&=\{a,ab^{31},b^{31},ab^{10},b^{10},b^{21},ab^{21}\}\\
B'&=\{e,b^6,ab^{25},ab^6,b^{25},b^{12},ab^{19},ab^{12},b^{19}\}.
\end{align*}
Choosing $h_2=ab^{31}$ and $f_2=f_{3,31}$ we have:
\begin{align*}
f_2(A')h_2&=\{e,b,b^2,b^3,a,ab,ab^2\}\\
h_2^{-1}f_2(B')&=\{b^4,b^{11},b^{18},b^{25},ab^3,ab^{10},ab^{17},ab^{24},ab^{31}\}.
\end{align*}
This is the canonical form of $(A',B')$.
 
Thus 
\begin{equation}
\label{f1f2.eq}
\big(f_1(A)h_1,h_1^{-1}f_1(B)\big)
=
\big(f_2(A')h_2,h_2^{-1}f_2(B')\big)
\end{equation}
and it follows that these two near-factorizations of $D_{32}$ are equivalent.
$\hfill\blacksquare$
\end{example}

Referring to the two near-factorizations from Example \ref{32.exam}, 
we show how to directly map $(A,B)$ to $(A',B')$.
From (\ref{f1f2.eq}), it follows easily that 
\begin{align*}
f_2(A') &= f_1(A)h_1h_2^{-1}\\
A' &= (f_1 \circ f_2^{-1})(A) f_2^{-1}(h_1h_2^{-1})\end{align*}
and
\begin{align*}
f_2(B') &= h_2h_1^{-1} f_1(B) \\
B' &=  f_2^{-1}(h_2h_1^{-1})(f_1 \circ f_2^{-1})(B).
\end{align*}
Denote $f = f_1 \circ f_2^{-1}$ and $h = f_2^{-1}(h_1h_2^{-1})$.
Observe that
\begin{align*}
h^{-1} &= (f_2^{-1}(h_1h_2^{-1}))^{-1}\\
&= f_2^{-1}(h_2h_1^{-1}),
\end{align*}
since $f(x^{-1}) =(f(x))^{-1}$ for any automorphism $f$ and any $x$.
Hence, $(A',B') = \Phi_{f,h}(A,B)$.

We have $f_1=f_{1,0}$, $h_1=ab^3$,  $f_2=f_{3,31}$ and $h_2=ab^{31}$.
Then
\begin{align*}
f_1 \circ f_2^{-1} &= f_{1,0} \circ (f_{3,31})^{-1}\\
& = f_{1,0} \circ (f_{11,11})\\
& = f_{11,11}
\end{align*}
and 
\begin{align*}
f_2^{-1}(h_1h_2^{-1}) &= f_2^{-1}(ab^{3}ab^{31})\\
&= f_2^{-1}(b^{28})\\
&= f_{11,11}(b^{28})\\
&= b^{28 \times 11}\\
&= b^{20}.
\end{align*}
In other words,
\begin{align*}
\Phi_{f_{11,11},b^{20}}\big((A,B)\big) &= (A',B').
\end{align*}
Then, from Lemma~\ref{Phi Inverse}, we have
\begin{align*}\Phi_{f_{3,31},b^{4}}\big((A',B')\big) &= (A,B).
\end{align*}

However, it turns out (see Remark \ref{k=3.rem}) that there is a simpler direct equivalence mapping between $(A,B)$ and $(A',B)$:
\begin{equation}
\label{simple.eq}
\Phi_{f_{21,0},e}\big((A,B)\big) = (A',B') \quad 
\text{and} \quad 
\Phi_{f_{29,0},e}\big((A',B')\big) = (A,B).
\end{equation}

As we mentioned in Section \ref{32.sec}, there are two direct constructions for infinite classes of near-factorizations of dihedral groups that can be found in the literature, namely in \cite{BHS} and \cite{CGHK}.
In describing the constructions, we make use of the following notation.
Let $n = 2^{2k-1}$, $\alpha = 2^k-1$, $d_1 = 2^k + 3$, and 
$d_2 = 2^{k+1} - 3$. Further, let $a_0 = 2^{k-1}-1$, $b_0 = 2^{k-1} = a_0+1$ and $s = 2^{k-1} - 3$.

We first present the construction from \cite{CGHK}.

\begin{lemma}
\rm{\cite{CGHK}}
Define \[A = \left\{b, b^2, b^3,  \dots , b^{a_0}\right\}
\bigcup \left\{ab, ab^2, ab^3,  \dots , ab^{a_0}\right\}
\bigcup \{a\}\]
and 
\[
B= \left\{b^{\alpha}, b^{2{\alpha}}, b^{3{\alpha}}, \dots , b^{b_0{\alpha}}\right\}
\bigcup \left\{ab^{\alpha}, ab^{2{\alpha}}, ab^{3{\alpha}},  \dots , ab^{b_0{\alpha}}\right\}
\bigcup \{e\}.\]
Then $(A,B)$ is a $(2^k-1,2^k+1)$-near-factorization of $D_n$.
\end{lemma}

Now we recall the construction from \cite{BHS}.

\begin{lemma}\rm{\cite[Theorem 2]{BHS}}
Define \[A' = \left\{b^{-s}, b^{-s+d_1}, \dots , b^{-s+(a_0-1)d_1}\right\}
\bigcup \left\{ab^{-s}, ab^{-s+d_1}, \dots , ab^{-s+(a_0-1)d_1}\right\}
\bigcup \{a\}\]
and 
\[
B' = \left\{b^{-s+d_2}, b^{-s+2d_2}, \dots , b^{-s+b_0d_2}\right\}
\bigcup \left\{ab^{-s+d_2}, ab^{-s+2d_2}, \dots , ab^{-s+b_0d_2}\right\}
\bigcup \{e\}.\]
Then $(A',B')$ is a $(2^k-1,2^k+1)$-near-factorization of $D_n$.
\end{lemma}

We will prove that $(A,B)$ and $(A',B')$ are equivalent near-factorizations.
First, we observe that  
\begin{align*} \alpha  d_1 - d_2 &= (2^k-1)(2^k+3) - (2^{k+1} - 3)\\
&= 2^{2k} +  2(2^k) - 3 - (2^{k+1}-3)\\ 
&= 2^{2k}\\
&= 2n,
\end{align*} so $\alpha  d_1 \equiv d_2 \bmod n.$ 
This means that $ \alpha \equiv d_2(d_1)^{-1} \bmod n,$ which is important in specifying the equivalence mapping. 

We will use the automorphism $f = f_{-d_1,0}$. 
In what follows, we make use of the fact that $s \equiv a_0d_1 \bmod n$ (equivalently, $s(d_1)^{-1} \equiv a_0 \bmod n$). This is proven as follows:
\begin{align*}
a_0d_1 - s  &= (2^{k-1}-1)(2^k+3) - (2^{k-1} - 3)\\
&= 2^{2k-1} - 2^k + 3(2^{k-1}) - 3 - 2^{k-1} + 3\\
&= 2^{2k-1} \\
& = n.
\end{align*}

We have 
\begin{align*} f(A) &=  \left\{b^{-d_1}, b^{-2d_1}, b^{-3d_1},  \dots , b^{-a_0d_1}\right\}
\bigcup \left\{ab^{-d_1}, ab^{-2d_1}, ab^{-3d_1},  \dots , ab^{-a_0d_1}\right\}
\bigcup \{a\}.
\end{align*}
Observe that
\begin{align*}
-d_1 &\equiv -s + (a_0-1)d_1 \bmod n\\
-2d_1 &\equiv -s + (a_0-2)d_1 \bmod n\\
\vdots \\
-a_0d_1 &\equiv -s \bmod n.
\end{align*}
Hence,
\begin{align*} f(A) &=  \left\{b^{-s + (a_0-1)d_1}, b^{-s + (a_0-2)d_1},   \dots , b^{-s}\right\}
\bigcup \left\{ab^{-s + (a_0-1)d_1}, ab^{-s + (a_0-2)d_1},   \dots , ab^{-s}\right\}
\bigcup \{a\}\\
&= A'.
\end{align*}
Also, 
\begin{align*} f(B) &= 
\left\{b^{-d_1\alpha}, b^{-2d_1\alpha}, b^{-3d_1{\alpha}}, \dots , b^{-b_0{d_1\alpha}}\right\}
\bigcup \left\{ab^{-d_1\alpha}, ab^{-2d_1{\alpha}}, ab^{-3d_1{\alpha}},  \dots , ab^{-b_0d_1{\alpha}}\right\}
\bigcup \{e\}
\end{align*}
Since $\alpha d_1 \equiv d_2 \bmod n$, we have
\begin{align*} f(B) &= 
\left\{b^{-d_2}, b^{-2d_2}, b^{-3d_2}, \dots , b^{-b_0d_2}\right\}
\bigcup \left\{ab^{-d_2}, ab^{-2d_2}, ab^{-3d_2},  \dots , ab^{-b_0d_2}\right\}
\bigcup \{e\}\\
&= 
\left\{b^0, b^{-d_2}, b^{-2d_2}, b^{-3d_2}, \dots , b^{-b_0d_2}\right\}
\bigcup \left\{ab^{-d_2}, ab^{-2d_2}, ab^{-3d_2},  \dots , ab^{-b_0d_2}\right\}.
\end{align*}

Now we make use of the fact that $s \equiv (b_0+1)d_2 \bmod n$. This is proven as follows:
\begin{align*}
(b_0+1)d_2 - s  &= (2^{k-1} + 1)(2^{k+1} - 3) - (2^{k-1} - 3)\\
&= 2^{2k} + 2^{k+1} - 3(2^{k-1}) - 3 - 2^{k-1} + 3\\
&= 2^{2k} \\
& = 2n.
\end{align*}
It therefore follows that
\begin{align*}
0 &\equiv -s + (b_0+1)d_2 \bmod n\\
-d_2 &\equiv -s + b_0d_2 \bmod n\\
-2d_2 &\equiv -s + (b_0-1)d_2 \bmod n\\
\vdots \\
-(b_0-1)d_2 &\equiv -s + 2d_2 \bmod n\\
-b_0d_2 &\equiv -s +d_2 \bmod n.
\end{align*}

Hence, 
\begin{align*}
B' &= \left\{b^{-s+d_2}, b^{-s+2d_2}, \dots , b^{-s+b_0d_2}\right\}
\bigcup \left\{ab^{-s+d_2}, ab^{-s+2d_2}, \dots , ab^{-s+b_0d_2}\right\}
\bigcup \left\{b^{-s+(b_0+1)d_2}\right\}
\\
&= \left\{b^{-s+d_2}, b^{-s+2d_2}, \dots , b^{-s+b_0d_2}, b^{-s+(b_0+1)d_2 }\right\}
\bigcup \left\{ab^{-s+d_2}, ab^{-s+2d_2}, \dots , ab^{-s+b_0d_2}\right\}\\
&= \left\{b^{-b_0d_2}, b^{-(b_0-1)d_2}, \dots , b^{-d_2}, b^{0 }\right\}
\bigcup \left\{ab^{-b_0d_2}, ab^{-(b_0-1)d_2}, \dots , ab^{-d_2}\right\}\\
&= f(B).
\end{align*}
Therefore, we have proven the following theorem.

\begin{theorem} The two near-factorizations $(A,B)$ and $(A',B')$ of $D_{2^{k-1}}$, $k \geq 3$ are equivalent. 
Specifically, \[
\Phi_{f_{-d_1,0},e}\big((A,B)\big) = (A',B').
\]
\end{theorem}

\begin{remark}
\label{k=3.rem}
When $k = 3$, we have $n = 32$ and $d_1 = 11$, so $-d_1 = 21$. The equivalence mapping in this case is the same as the equivalence mapping (\ref{simple.eq}). 
\end{remark}

In \cite{BHS}, a graph $P = P(\Gamma,A,B)$ is constructed from a near-factorization 
$(A,B)$ in a finite group $\Gamma$ as follows. The vertices of $P$ are the elements of $\Gamma$. Let us define a hypergraph $R := \{gA : g \in \Gamma\}$. Finally, let two vertices 
$x$, $y$ of $P$ be adjacent if and only if some $gA \in R$
contains both $x$ and $y$.

\begin{theorem}\label{LemP}
If the near-factorizations $(A,B)$ and $(A',B')$ of the 
finite group  $\Gamma$ are equivalent, then the graphs
$P = P (\Gamma, A, B)$ and $P' = P (\Gamma, A', B')$
are isomorphic.
\end{theorem}

\begin{proof}
Suppose the near-factorizations $(A,B)$ and $(A',B')$ of the 
finite group  $\Gamma$ are equivalent. Then there exists
$f \in \Aut(\Gamma)$  and $h \in \Gamma$ such that
\[
A'= f(A)h \quad 
\text{ and  } \quad
B'= h^{-1}f(B).
\]
Let 
$R = \{xA : x \in \Gamma\}$
and
$R' = \{xA' : x \in \Gamma\}$ 
be the hypergraphs 
that are associated with $(A,B)$ and $(A',B')$, respectively.
Define $\theta:\Gamma\rightarrow \Gamma$ by
$\theta(x)= f(x)h$. Then 
\[
\theta(A)=\{\theta(x): x \in A\} 
=
\{f(x)h: x \in A\} 
=f(A)h
= A'
.\]
Further,  for any $x,y\in\Gamma$, it holds that 
\begin{equation}
\label{theta.eq}
\theta(xy)= f(xy)h = f(x)f(y)h= f(x)\theta(y).
\end{equation}

Now consider two vertices $x_1$ and $x_2$ of
the graph $P$. 
Suppose $x_1$ and $x_2$ are adjacent in $P$.
Then there exists $g \in G$
such that $x_1 = ga_1$ and $x_2=ga_2$ for some $a_1,a_2 \in A$.
Let $g'=f(g)$.
Then for $i=1,2$, we have from (\ref{theta.eq}) that
\[
\theta(x_i)=\theta(ga_i) = f(g)\theta(a_i) = g'f(a_i)h.\]
Note that $f(a_1)h, f(a_2)h \in A' = f(A)h$.
Therefore $\theta(x_1)$ and $\theta(x_2)$ are adjacent in $P'$.

To prove the converse, suppose that $y_1$ and $y_2$ are adjacent in $P’$.
Hence, $y_i = g a_i'$, for $i = 1,2$, where $a_1', a_2' \in A’ = f(A)h$.
We have $a_i’ = f(a_i) h$ for some $a_i \in A$, $i = 1,2$.
So $y_i = g f(a_i) h$, for $i = 1,2$.

We now apply $\theta^{-1}$. Note that  $\theta^{-1}(y) = f^{-1}(yh^{-1})$ for any $y \in \Gamma$.
Hence \[\theta^{-1}(y_i)  = f^{-1}(g f(a_i)) = f^{-1}(g) f^{-1}(f(a_i)) =  f^{-1}(g)  a_i.\]
It follows that  $\theta^{-1}(y_1)$  and $\theta^{-1}(y_2)$ are both in  $f^{-1}(g) A$.
Therefore they are adjacent in $P$.
\end{proof}

The following notation is from \cite{BHS}:
\begin{quotation}
Given a graph $G$, $\alpha = \alpha(G)$ denotes the size of the largest stable (or independent) set of vertices of $G$ while $\omega = \omega(G)$ is the size of its largest clique. An $w$-clique is a clique with $\omega$ vertices.
The edge $e=xy$ is \emph{$\alpha$-critical} in $G$ if $\alpha(G-e)>\alpha(G)$. The pair $xy$ of
vertices is an \emph{$\omega$-critical non-edge} if the edge $xy$ is $\alpha$-critical in $G$. Obviously, this is equivalent to each of the following properties: i) $\omega(G + xy) > \omega(G)$; and ii) there exists a pair of cliques $(X, Y )$ such that $X - Y = \{x\}$, $Y - X = \{y\}$. 
\end{quotation}

The authors of~\cite{BHS} provide  two definitions
concerning alternating graphs.

\begin{definition}[{Bocs\'{o} \textit{et al.}, \cite[Definition 1]{BHS}}] 
The graph $P$ is \emph{alternating} if its $\alpha$-critical edges have no common endpoints, 
and the same is true for their $\omega$-critical non-edges and both ``matchings''
cover the whole vertex set of the graph.
\end{definition}

\begin{definition}[{Bocs\'{o} \textit{et al.}, \cite[Definition 2]{BHS}}]\label{DefAlt}
A near-factorization $(A, B)$ in the group $\Gamma$ is \emph{alternating} if the graph 
$P = P (\Gamma, A, B)$ is alternating.
\end{definition}

\begin{corollary}
Two equivalent near-factorizations of a finite group $G$
are either both alternating or both non-alternating.
\end{corollary}
\begin{proof}
This follows from  Lemma~\ref{LemP} and Definition~\ref{DefAlt} .
\end{proof}

We conclude
that the near-factorizations of $D_{2^{2k-1}}$
found in ~\cite{BHS} and ~\cite{CGHK} are either both alternating or both non-alternating,
contrary to the claims in~\cite{BHS}.

\subsection{Some enumerations of nonequivalent near-factorizations of $D_n$}

Using the computer, we have enumerated 
all nonequivalent near-factorizations of $D_n$ for $n \leq 32$. We used the technique described in \cite{KPS}, adapted to the setting of dihedral groups, to 
obtain the following result.

\begin{theorem}
\label{enum.thm}
For all $n \leq 32$ and all proper divisors $k$ of $2n-1$, there is a unique $(k,(2n-1)/k)$-near-factorization of $D_n$ up to equivalence.
\end{theorem}

\begin{remark}
The solutions mentioned in Theorem \ref{enum.thm} are of course equivalent to the ones that can also be obtained from \cite[Example 2]{CGHK}.
\end{remark}

\section{Near-factorizations in dihedral groups derived from near-factorizations in cyclic groups}
\label{pech.sec}

For some values of $n > 32$, nonequivalent near-factorizations of $D_n$ exist. We will construct  nonequivalent near-factorizations of $D_{41}$  and $D_{95}$ in Sections \ref{D41.sec} and \ref{D95.sec}, resp. These constructions use a technique described in \cite{Pech}.

Suppose that $n$ is odd; then $\zed_{2n}$ is isomorphic to $\zed_n \times \zed_2$. It is shown in \cite[Corollary 1]{Pech} that a symmetric near-factorization in 
$\zed_{2n}$ for $n$ odd gives rise to a strongly symmetric near-factorization in $D_n$. 
We will call this the \emph{P\^{e}cher transform}, which we describe now. 

\begin{enumerate}
\item Suppose $(A,B)$ is a symmetric near-factorization of $Z_{2n}$ where $n$ is odd.
\item For every $x \in \zed_{2n}$, define $\phi(x) = (x \bmod 2, x \bmod n)$. It is clear that 
$\phi : \zed_{2n} \rightarrow \zed_2 \times \zed_n$ is a group isomorphism.
Applying $\phi$ to $(A,B)$, we obtain the symmetric near-factorization 
$(A^*,B^*) = (\phi(A), \phi(B))$ in  $\zed_2 \times \zed_n$. 
\item For every $(i,j) \in \zed_2 \times \zed_n$, define $\psi(i,j) = a^ib^j$, where $a$ and $b$ are generators of 
$D_n$ as defined in (\ref{dihedral.eq}) (of course $\psi$ is not a group isomorphism). Thus $\psi : \zed_2 \times \zed_n \rightarrow D_n$. Applying $\psi$ to  $(\phi(A), \phi(B))$, we obtain $(A',B')$, which turns out to be a strongly symmetric near-factorization of $D_n$. 
\end{enumerate}

Figure \ref{pecher.fig} illustrates the steps in the P\^{e}cher transform.
\begin{figure}
\[
\begin{array}{ccccc}
\zed_{2n} & \RAR{.75in}{\phi} & \zed_2 \times \zed_n & \RAR{.75in}{\psi} & D_n
\\
x & \RAR{.75in}{} &  (x \bmod 2, x \bmod n) = (i,j)
& \RAR{.75in}{} & a^ib^j
\\
(A,B) & \RAR{.75in}{} & (A^*,B^*) & \RAR{.75in}{} & (A',B')
\end{array}
\]
\caption{The P\^{e}cher transform}
\label{pecher.fig}
\end{figure}

\begin{example}
$A = \{0,1,9\}$ and $B = \{2,5,8\}$ form a symmetric near-factorization of $\zed_{10}$.
Applying $\phi$, the associated (symmetric) near-factorization of $\zed_{2} \times \zed_5$ is
\[(A^*,B^*) = (\{(0,0),(1,1), (1,4)\}, \{(0,2),(1,0), (0,3)\}).\]
Then we use $\psi$ to obtain the strongly symmetric near-factorization 
\[(A',B') = \left(\left\{e,ab, ab^4\right\}, \left\{b^2,a, b^3\right\}\right)\] of $D_5$.
$\hfill\blacksquare$
\end{example}

We review the proof from \cite{Pech} that $(A',B')$, is a near-factorization of $D_n$. First consider an element $b^j \in D_n$. We want to express $b^j$ as a product of an element from $A'$ and an element from $B'$. In the near-factorization $(A^*,B^*)$, the value $(0,j)$ is obtained as a sum in one of the two following ways:
\begin{enumerate}
\item $(0,j) = (0,k) + (0,\ell)$, where $j = k + \ell$, $(0,k) \in \phi(A)$ and $ (0,\ell) \in \phi(B)$, or
\item $(0,j) = (1,k) + (1,\ell)$, where $j = k + \ell$, $(1,k) \in \phi(A)$ and $ (1,\ell) \in \phi(B)$.
\end{enumerate}
In the first case, we have $b^k \in A'$ and $b^{\ell} \in B'$, and $b^k b^{\ell} = b^j$.
In the second case, we have $ab^k \in A'$ and $ab^{\ell} \in B'$. Because  $\phi(A)$ is symmetric, we also have
$(1,-k) \in \phi(A)$ and hence $ab^{-k} \in A'$. Then we obtain the product
$(ab^{-k}) (ab^{\ell}) = b^{k + \ell} = b^j$.

Now we look at elements of the form $ab^j$. In the near-factorization $(A^*,B^*)$, the value $(1,j)$ is obtained in one of the two following ways:
\begin{enumerate}
\item $(1,j) = (0,k) + (1,\ell)$, where $(0,k) \in \phi(A)$ and $ (1,\ell) \in \phi(B)$, or
\item $(1,j) = (1,k) + (0,\ell)$, where $(1,k) \in \phi(A)$ and $ (0,\ell) \in \phi(B)$.
\end{enumerate}
In the first case, we have $b^k \in A'$ and $ab^{\ell} \in B'$.
Because  $\phi(A)$ is symmetric, we also have
$(0,-k) \in \phi(A)$ and hence $b^{-k} \in A'$.
Then we obtain the product
$(b^{-k}) (ab^{\ell}) = a b^{k + \ell} = ab^j$.
In the second case, we have $ab^k \in A'$ and $b^{\ell} \in B'$.  Then we obtain the product
$(ab^{k}) (b^{\ell}) = ab^{k+\ell} = ab^{j}$.

Because $D_{n} - e \subseteq A'B'$ and $|A'| \times |B'| = |A| \times |B| = 2n-1$, we conclude that 
$D_{n} - e = A'B'$ and therefore $(A',B')$ is a near-factorization of $D_n$. 

Finally, it is obvious that $(A',B')$ is strongly symmetric. The 
two elements $(i,j)$ and $(i,-j)$ in $A^*$ or in $B^*$ ($j \neq 0$) are mapped to
$a^ib^j$ and $a^ib^{-j}$. Also, an element $(i,0)$ ($i = 0,1)$ is mapped to $a^i$.

Therefore we have established the following theorem.
\begin{theorem} 
\label{Pech.thm}
Suppose $(A,B)$ is a symmetric near-factorization of $\zed_{2n}$, where $n$ is odd.
Let $(A',B')$ be obtained from $(A,B)$ by applying the P\^{e}cher transform. Then $(A',B')$ is a 
strongly symmetric near-factorization of $D_{n}$.
\end{theorem}


We now show that, if $(A,B)$ and $(C,D)$ are equivalent {symmetric} near-factorizations of $\zed_{2n}$ (where $n$ is odd), then  
$(A',B')$ and $(C',D')$ are equivalent {strongly symmetric} near-factorizations of $D_n$. From Theorem \ref{Pech.thm}, we know that $(A',B')$ and $(C',D')$ are both {strongly symmetric} near-factorizations of $D_n$; we just need to show that they are equivalent.

We first state and prove a numerical lemma.

\begin{lemma}
\label{numerical.lem}
Suppose $(A,B)$ and $(C,D)$ are   equivalent {symmetric} near-factorizations of $\zed_{2n}$ where $n$ is odd.
Suppose also that  $C = rA+h$ and $D = rB - h$.
Then $\GCD(r,2n) = 1$ and $h = 0$ or $n$.
\end{lemma}

\begin{proof} The mapping $x \mapsto rx$ must be an automorphism of $\zed_{2n}$, so $\GCD(r,2n) = 1$.
Hence, it suffices to prove that  $h = 0$ or $n$.
 For $x,y \in \zed_{2n}$, define $x \sim y$ if $x = \pm y$. There are two equivalence classes of size one, namely
 $0$ and $n$, and all other equivalence classes have size two and consist of $\{x, -x\}$, where $x \neq 0,n$.
 
Recall that $|A|$ and $|B|$ are both odd, since $|A| \times |B| = 2n-1$. (Also, $|C| = |A|$ and $|D| = |B|$.)
Because $A$ is symmetric, it is a union of equivalence classes.  $|A|$ is odd and there are two equivalence classes of size one, so $A$ contains exactly one equivalence class of size one. Similarly, $B$ contains exactly one equivalence class of size one. 
Because $|A|$ and $|B|$ are disjoint and both of them are symmetric, it follows that
$0 \in A$ and $n \in B$, or vice versa.

We have $C = rA+h$ and $D = rB-h$. Clearly $rA$ and $rB$ are each a union of equivalence classes, because the function $x \mapsto rx$ maps equivalence classes to equivalence classes. So we have $0 \in rA$ and $n \in rB$, or vice versa. 

Without loss of generality, suppose $0 \in rA$ (so $n \in rB$). Then it is clear that 
\[ \sum_{x \in rA} x = 0 \quad \text{and} \quad \sum_{x \in rB} x = n.\]
Now consider \[S = \sum_{x \in C} x.\] The set $C$ is symmetric, so we have $S \equiv 0 \text{ or } n \bmod 2n$. In either case, $S \equiv 0 \bmod n$.
However, it is also the case that 
\begin{align*}
S &\equiv \sum_{x \in rA} (x+h) \bmod 2n\\
&\equiv \sum_{x \in rA} x + \sum_{x \in rA} h \bmod 2n\\
&\equiv |A|h \bmod 2n.
\end{align*} Therefore, $|A|h \equiv 0 \bmod n$.
From the relation $|A| \times |B| = 2n-1$, it follows that $\GCD(|A|,n) = 1$. Hence, $h \equiv 0 \bmod n$ and therefore $h = 0$ or $n$.
\end{proof}

\begin{remark}
The fact that $\GCD(|A|,n) = 1$ is essential to the proof of Lemma \ref{numerical.lem}. For, consider the set $A = \{ 0,2,12,4,10,6,8\}$, which is a symmetric subset of $\zed_{14}$. If we take $h = 5$, then $A+h = \{ 7, 1,13,3,11,5,9\}$ is also symmetric. So we have a translate of a symmetric subset by a value $h \neq n \text{ or } 0$ modulo $2n$, which also yields a symmetric subset. Here $|A| = 7$, so $\GCD(|A|,2n) = 7 > 1$, and of course there is no contradiction to Lemma \ref{numerical.lem}.
\end{remark}

We will also make use of the following similar lemma later in this section.
\begin{lemma}
\label{lem:var}
Suppose $A \subseteq \zed_n$  is {symmetric}, where $n$ is odd. Suppose $C = rA+h$ where $\GCD(r,n) = 1$ and $h \in \zed_n$. If $C$ is symmetric, then
$h$ is a multiple of $n / \GCD(|A|,n)$.
\end{lemma}

\begin{proof}
For $x,y \in \zed_{n}$, define $x \sim y$ if $x = \pm y$. Because $n$ is odd, there is one equivalence class of size one (namely, $0$), and all other equivalence classes have size two and consist of $\{x, -x\}$, where $x \neq 0$. If follows easily that the sum of the elements in a symmetric subset is $0$ modulo $n$. Hence, the sum of the elements in $A$ is $0$ modulo $n$, as is the sum of the elements in $rA$.

Define \[S = \sum_{x \in C} x.\] The set $C$ is symmetric by assumption, so we have $S\equiv 0 \bmod n$. On the other hand, 
\begin{align*}
S &\equiv \sum_{x \in rA} (x+h) \bmod n\\
&\equiv \sum_{x \in rA} x + \sum_{x \in rA} h \bmod n\\
&\equiv |A|h \bmod n.
\end{align*}
Therefore 
\[ |A|h \equiv 0 \bmod n\]
and hence $h$ is a multiple of $n / \GCD(|A|,n)$.
\end{proof}

\begin{remark}
If $\GCD(|A|,n) = 1$ in Lemma \ref{lem:var}, then it follows that $h = 0$.
\end{remark}

We now prove the equivalence of $(A',B')$ and $(C',D')$, which are obtained from $(A,B)$ and $(C,D)$ using the  P\^{e}cher transform.
Because $\phi$ is a group isomorphism, it is clear that $(A^*,B^*)$ and $(C^*,D^*)$ are equivalent.
From Lemma \ref{numerical.lem}, we $h = 0$ or $n$. Therefore  $\phi(h) = (1,0) \text{ or } (0,0)$, so $-\phi(h) = \phi(h)$. 
Then 
\begin{equation}
\label{C*.eq}
C^*= f(A^*) + \phi(h) \quad \text{and} \quad D^* = \phi(h) + f(B^*),
\end{equation} where
$f \in \Aut(\zed_{2} \times \zed_n)$ and $\phi(h) = (1,0) \text{ or } (0,0)$.
Note that $f$ has the form $f(i,j) = (i,rj)$ where $\GCD (r,2n)=1$.

Define $f' = f_{r,0} \in \Aut(D_n)$.
First, from the following commutative diagram,
we see that $\psi \circ f' = f \circ \psi$:

\begin{center}
\begin{tikzcd}
(i,j) \arrow[r, "\psi"] \arrow[d, "f"]
& a^ib^j \arrow[d, "f' = f_{r,0}" ] \\
(i,rj) \arrow[r,  "\psi" ]
&  a^ib^{rj}
\end{tikzcd}
\end{center}

Then 
\begin{align*}
f'(A') &= f_{r,0}(\psi(A^*)) \\
&= \psi (f(A^*)). 
\end{align*}

Recall that $\phi(h) = (1,0) \text{ or } (0,0)$.  First, suppose $\phi(h) = (1,0)$.
Consider  $(i,j) \in A^*$. Because $A^*$ is symmetric, we have $(i,-j) \in A^*$.
Then, from (\ref{C*.eq}), we have 
\begin{equation}
\label{CC*.eq}
\{ (i+1, rj), (i+1, -rj)\} \subseteq C^*.
\end{equation}

Corresponding to the elements $(i,j), (i,-j) \in A^*$, we have $\psi(i,j), \psi(i,-j) \in A'$.
Define $h' = \psi(h)$; then $h' = a$. 
We have
\begin{align}
\left\{f_{r,0}(\psi(i,j)), f_{r,0}(\psi(i,-j))\right\}h' &= \left\{ a^ib^{rj}, a^ib^{-rj}\right\}a \nonumber\\
&= \left\{ a^{i+1}b^{-rj }, a^{i+1}b^{ rj}\right\}. \label{111.eq}
\end{align}

The  elements in $C^*$ were given in (\ref{CC*.eq}).
Applying $\psi$, the corresponding elements in $C'$ are
\begin{align*}
\{\psi (i+1, rj), \psi(i+1, -rj)\} &= \{a^{i+1}b^{rj}, a^{i+1}b^{-rj} \},
\end{align*}
which agrees with (\ref{111.eq}). So we have proven that $C'= f'(A')  h'$.

\bigskip

The argument that $D'= (h')^{-1}f'(B')$ is similar. Define $f$, $f'$, $h$ and $h'$ as before.
Note that  $(h')^{-1} = a = h'$.

Consider  $(i,j) \in B^*$. Because $B^*$ is symmetric, we have $(i,-j) \in B^*$.
Then, from (\ref{C*.eq}), we have 
\begin{equation}
\label{DD*.eq}\{ (i+1, rj), (i+1, -rj)\} \subseteq D^*.
\end{equation}

Corresponding to the elements $(i,j), (i,-j) \in B^*$, we have $\psi(i,j), \psi(i,-j) \in B'$.
We have
\begin{align}
a\left\{f_{r,0}(\psi(i,j)), f_{r,0}(\psi(i,-j))\right\} &= a\left\{ a^ib^{rj}, a^ib^{-rj}\right\}\nonumber\\
&= \left\{ a^{i+1}b^{rj}, a^{i+1}b^{-rj}\right\}\label{112.eq}.
\end{align}

The  elements in $D^*$ were given in (\ref{DD*.eq}). Applying $\psi$, the corresponding elements in $D'$ are
\begin{align*}
\{\psi (i+1, rj), \psi(i+1, -rj)\} &= \{a^{i+1}b^{rj}, a^{i+1}b^{-rj} \},
\end{align*}
which agrees with (\ref{112.eq}). So we have proven that $C'= f'(A')  h'$.

The other possible case, where $h= (0,0)$ and $h' = e$, is similar. The details are left for the reader to verify.
Summarizing, we have proven the following.
\begin{theorem}
\label{equiv1.thm}
 Suppose that $(A,B)$ and $(C,D)$ are equivalent symmetric near-factorizations of $\zed_{2n}$, where $n$ is odd. Let $(A',B')$ and $(C',D')$ be obtained from  $(A,B)$ and $(C,D)$ by applying the P\^{e}cher transform.
Then  $(A',B')$ and $(C',D')$ are equivalent strongly symmetric near-factorizations of $D_{n}$.
\end{theorem}

We now consider the inverse P\^{e}cher transform. We basically reverse the operations depicted in Figure \ref{pecher.fig}. Suppose we start with a strongly symmetric near-factorization of $D_{n}$, say $(A',B')$. 
As before, we assume that $n$ is odd.
First we apply $\psi^{-1}$, which replaces every element $a^ib^j \in D_n$ with  $(i,j) \in \zed_2 \times \zed_n$. Then we apply $\phi^{-1}$ to obtain an element $x \in \zed_{2n}$. The mapping $\phi^{-1}$ is just an application of the Chinese remainder theorem. That is, $\phi^{-1}(i,j) = n   d_2 \, i + 2  d_1 \, j$, where $d_1$ and $d_2$ are integers such that $2d_1 + n d_2 = 1$. (Note that $d_1$ and $d_2$ can be obtained from the Euclidean algorithm; it is easily verified that $d_1=(n+1)/2$ and $d_2=-1$.) Of course $\phi^{-1}$ is an isomorphism from $\zed_2 \times \zed_n$ to $\zed_{2n}$. It is convenient  just to consider the image of  $(A',B')$ under $\psi^{-1}$, which we denote, as before, by $(A^*, B^*)$.

We want to show that $(A^*, B^*)$ is a near-factorization. First, consider an element $(0,j) \in \zed_2 \times \zed_n$. We want to express $(0,j)$ as a sum of an element in $A^*$ and an element in $B^*$.
The value $b^j$ occurs as a product of an element from $A'$ and an element from $B'$. There are two ways in which this can happen:
\begin{enumerate}
\item $b^j = b^k b^{\ell}$, where $b^k \in A'$, $b^{\ell} \in B'$ and $k + \ell = j$, or
\item $b^j = (ab^k) (ab^{\ell})$, where $ab^k \in A'$, $ab^{\ell} \in B'$ and $\ell - k = j$.
\end{enumerate}
In the first case, we have $(0,k)\in A^*$ and $(0,\ell) \in B$, and $(0,j) = (0,k)+(0,\ell)$.
In the second case, we make use of the fact that $(A^*, B^*)$ is strongly symmetric. Because  $ab^k \in A'$, we also have $ab^{-k} \in A'$. Therefore $(1,-k)\in A^*$ and $(1,\ell) \in B^*$, and $(0,j) = (1,-k)+(1,\ell)$.

Now we look at elements $(1,j) \in \zed_2 \times \zed_n$. 
The value $ab^j$ occurs as a product of an element from $A'$ and an element from $B'$. There are two ways in which this can happen:
\begin{enumerate}
\item $ab^j = (ab^k) b^{\ell}$, where $ab^k \in A'$, $b^{\ell} \in B'$ and $k + \ell = j$, or
\item $ab^j = b^k (ab^{\ell})$, where $b^k \in A'$, $ab^{\ell} \in B'$ and $\ell - k = j$.
\end{enumerate}
In the first case, we have $(1,k)\in A^*$ and $(0,\ell) \in B$, and $(1,j) = (1,k)+(0,\ell)$.
In the second case, we make use of the fact that $(A^*, B^*)$ is strongly symmetric. Because  $b^k \in A'$, we also have $b^{-k} \in A'$. Therefore $(0,-k)\in A^*$ and $(1,\ell) \in B^*$, and $(1,j) = (0,-k)+(1,\ell)$.

Because $D_{n} - e \subseteq A^*B^*$ and $|A^*| \times |B^*| = |A'| \times |B'| = 2n-1$, we conclude that 
$D_{n} - e = A^*B^*$ and therefore $(A^*,B^*)$ is a near-factorization of $D_n$. 

Finally, it is obvious that $(A^*,B^*)$ is  symmetric. For $j \neq 0$, the
two elements $a^ib^j$ and $a^ib^{-j}$ in $A^*$ (or in $B^*$) ($j \neq 0$) are mapped by $\psi^{-1}$ to
$(i,j)$ and $(i,-j)$. Also, an element $a^i$ ($i = 0,1)$ is mapped to $(i,0)$.

Therefore we have established the following theorem, which we state in terms of $(A,B)$ and $(A',B')$.
\begin{theorem} 
\label{Pechinb.thm}
Suppose $(A',B')$ is a strongly symmetric near-factorization of $D_{n}$, where $n$ is odd.
Let $(A,B)$ be obtained from $(A',B')$ by applying the inverse P\^{e}cher transform. Then $(A,B)$ is a 
 symmetric near-factorization of $\zed_{2n}$.
\end{theorem}

Finally, we show that
equivalent strongly symmetric near-factorizations of $D_{n}$ are mapped to
equivalent symmetric near-factorizations of $\zed_{2n}$ using the inverse P\^{e}cher transform, provided that certain numerical conditions are satisfied. First, we prove a useful lemma.

\begin{lemma}\label{lem:h1aj0}
Let   $C_n$ denote the subgroup of $n$ rotations in $D_{n}$ and
let $R$ be a subset of the $n$ reflections in $D_n$.  Then $R$ is strongly symmetric if and only if $aR$ is a symmetric subset of $C_n$.
\end{lemma}
\begin{proof}
Reflections in $D_n$ have the form $ab^i$ for some $i$ and hence elements of $aR$ have the form $a(ab^i)=b^i$.
Therefore the set $aR$ consists of rotations.  We observe that left multiplication by $a$ gives a bijection between elements of $R$ and elements of $aR$.  The set $aR$ is symmetric if and only if $b^{-i} \in aR$ whenever $b^i\in aR$.  In this case we have $ab^{-i}\in R$ whenever $ab^i\in R$, and conversely.
It follows that $R$ is strongly symmetric if and only if $aR$ is symmetric.  
\end{proof}

\begin{lemma}
\label{hj.lem}
Suppose $(A',B')$ is a strongly symmetric near-factorization of $D_{n}$ with $|A'|=k$ and $|B'|=\ell$.
Suppose further that the equivalent near-factorization $(C',D')$ given by $C'=f_{i,j}(A')h$ and $D'=h^{-1}f_{i,j}(B')$ is also strongly symmetric. Finally, suppose  $\GCD(n, (k+1)/2) = \GCD(n, (\ell+1)/2) = 1$. 
Then $n$, $k$ and $\ell$ are all odd integers, 
$h\in \{e,a\}$ and $j=0$. 
\end{lemma}
\begin{proof}
We observe that $k\ell+1=2n$ as $(A',B')$ is a near-factorization of $D_{n}$, so it follows that both $k$ and $\ell$ are odd.  

We note that $n$ is odd, since if $(k\ell+1)/2$ were even, then we would have $k\ell\equiv 3 \pmod{4}$.  Then either $k\equiv 1\pmod{4}$ and $\ell \equiv 3 \pmod{4}$, in which case $\GCD(n,(\ell+1)/2) \geq 2$; or $k\equiv 3 \pmod{4}$ and $\ell \equiv 1 \pmod{4}$, in which case $\GCD(n,(k+1)/2) \geq 2$.

We begin by considering the case where $h$ is a rotation, i.e., $h=b^r$ for some $r$.  From Theorem \ref{size.thm}, we can assume without loss of generality that $A'$ contains $(k+1)/2$ rotations; let ${A'_0}$ be the set of rotations in $A'$.  Then ${A'_0}$ is a symmetric subset of $C_n = \langle b \rangle$. We also have that $f_{i,j}({A'_0})h$ is a symmetric subset of $C_n$, because $C' =  f_{i,j}(A')h$ is assumed to be symmetric.  Now, $f_{i,j}({A'_0})h={A'_0}^i b^r$.   Because $(k+1)/2$ is coprime to $n$, Lemma~\ref{lem:var} (restated in multiplicative form) implies $h=e$. 

Let ${B'_1}$ be the set of reflections in $B'$.  From Theorem \ref{size.thm}, we have $|{B'_1}|=(\ell+1)/2$.
Further,  ${B'_1}$ is a strongly symmetric set, so $a{B'_1}$ is a symmetric subset of $C_n$ by Lemma \ref{lem:h1aj0}.  Because $f_{i,j}$ is an automorphism, we have
\begin{align*}
f_{i,j}(a{B'_1})&=f_{i,j}(a)f_{i,j}({B'_1}),\\
&=ab^jf_{i,j}({B'_1}),\\
&=b^{-j}(af_{i,j}({B'_1})).
\end{align*}

As $f_{i,j}({B'_1})$ is strongly symmetric, we have that $af_{i,j}({B'_1})$ is symmetric  by Lemma \ref{lem:h1aj0}.  We observe that $f_{i,j}(a{B'_1})=(a{B'_1})^i$, since $a{B'_1}\subseteq C_n.$  Thus 
\[af_{i,j}({B'_1})=b^jf_{i,j}(a{B'_1})=b^j(a{B'_1})^i,\] and both $af_{i,j}({B'_1})$ and $a{B'_1}$ are symmetric.  Then Lemma~\ref{lem:var} (using multiplicative notation) tells us that $j=0$, because $(\ell+1)/2$ is coprime to $n$.

Now consider the case where $h$ is a reflection, so $h=ab^r$ for some $r$. We have 
\[ C' = f_{i,j}({{A'_0}})ab^r={{A'_0}}^iab^r=a{{A'_0}}^{-i}b^r=a{{A'_0}}^ib^r\] as ${{A'_0}}$ is a symmetric set of rotations.  Now $C' = f_{i,j}({{A'_0}})ab^r$ is strongly symmetric by assumption, so Lemma \ref{lem:h1aj0} implies that $af_{i,j}({{A'_0}})ab^r={{A'_0}}^ib^r$ is symmetric.  We therefore have
\[a C' = {{A'_0}}^ib^r,\] where $aC'$ and $A'_0$ are symmetric,  
so Lemma~\ref{lem:var} tells us that $r=0$, i.e., $h=a$.  

As ${B'_1}$ is a strongly symmetric set of reflections, we know from Lemma \ref{lem:h1aj0} that $a{B'_1}$ is symmetric.  
Furthermore, we have already noted that $af_{i,j}({B'_1})$ is also symmetric.   We have
\begin{align*}
af_{i,j}({B'_1})&=af_{i,j}(a^2{B'_1}),\\
&=af_{i,j}(a)f_{i,j}(a{B'_1}),\\
&=aab^j(a{B'_1})^i,\\
&=b^{j}(a{B'_1})^{i}.
\end{align*}
Once again we apply Lemma~\ref{lem:var} to deduce that $j=0$.
\end{proof}

\begin{theorem}
\label{equiv2.thm}
Suppose $(A^\prime,B^\prime)$ with $|A|=k$ and $|B|=\ell$ and $(C^\prime,D^\prime)$ are equivalent strongly symmetric near-factorizations of $D_{n}$ with $C^\prime=f_{i,j}(A^\prime)h$ and $D^\prime=h^{-1}f_{i,j}(B^\prime)$.  Suppose that $\GCD(n, (k+1)/2) = \GCD(n, (\ell+1)/2) = 1$. 
Let $(A,B)$ and $(C,D)$ be obtained from $(A^\prime,B^\prime)$ and $(C^\prime,D^\prime)$ by applying the inverse P\^{e}cher transform.  Then $(A,B)$ and $(C,D)$ are equivalent symmetric near factorizations of $\zed_{2n}$.
\end{theorem}

\begin{proof}
From Lemma \ref{hj.lem}, $n$ must be odd, $h\in\{e,a\}$ and $j=0$.
By Theorem \ref{Pechinb.thm}, we have that $(A,B)$ and $(C,D)$ are symmetric near-factorizations of $\mathbb{Z}_{2n}$, so it remains to prove they are equivalent. Since $\phi$ is a group isomorphism from $\zed_{2n}$ to $\zed_2 \times \zed_n$, it suffices to prove that 
$(A^*,B^*)$ and $(C^*,D^*)$ are equivalent.

Consider first the case where $h=e$.  Then 
\[C^\prime=f_{i,0}(A^\prime) \quad \text{and} \quad D^\prime=f_{i,0}(B^\prime),\]  
where $f_{i,0} \in \Aut(D_n)$.

Let $(A^\ast,B^\ast)$ denote the image of $(A^\prime,B^\prime)$ under $\psi^{-1}$.  
We observe that $\psi^{-1}(a^eb^h) = (e,h)$.
Further, $f_{i,0}(a^eb^h) = a^eb^{ih}$. Define the mapping $f\colon\mathbb{Z}_2\times\mathbb{Z}_n\rightarrow \mathbb{Z}_2\times\mathbb{Z}_n$ by $f(e,h) = (e,ih)$. Clearly $f$ is an automorphism of $\zed_2 \times \zed_n$. Then we have have the following commutative diagram, which illustrates the fact that $\psi^{-1} \circ f = f_{i,0} \circ \psi^{-1}$:
\begin{center}
\begin{tikzcd}
a^eb^{h} \arrow[r, "\psi^{-1}"] \arrow[d, "f_{i,0}"]
& (e,h) \arrow[d, "f" ] \\
a^eb^{ih} \arrow[r,  "\psi^{-1}" ]
&  (e,ih)
\end{tikzcd}
\end{center}
 Hence, $C^\ast=f(A^\ast)$ and $D^\ast=f(B^\ast)$, as desired.

\medskip

Now consider the case where $h=a$.  
Then 
\[C^\prime=f_{i,0}(A^\prime)a \quad \text{and} \quad D^\prime=af_{i,0}(B^\prime).\] 

Observe that 
\begin{equation}
\label{aleft.eq}
a f_{i,0}(a^eb^h) = a(a^eb^{ih}) = a^{e+1}b^{ih}
\end{equation} and 
\begin{equation}
\label{aright.eq}f_{i,0}(a^eb^h)a = (a^eb^{ih})a =  a^{e+1}b^{-ih}.\end{equation}
Recalling that $A'$ is strongly symmetric, consider a pair of elements $a^eb^h, a^eb^{-h} \in A'$. Under $\psi^{-1}$, this pair is mapped to the pair $(e,h), (e,-h) \in A^*$.
From (\ref{aright.eq}), the corresponding pair of elements in $C'$ is $a^{e+1}b^{-ih}, a^{e+1}b^{ih}$. Applying $\psi^{-1}$, we get the pair 
$(e+1,-ih), (e+1,ih) \in A^*$. Define
$f\colon\mathbb{Z}_2\times\mathbb{Z}_n\rightarrow \mathbb{Z}_2\times\mathbb{Z}_n$ by $f(e,h) = (e+1,ih)$.
Clearly $f$ is an automorphism of $\zed_2 \times \zed_n$ and
$f$ maps the pair $\{(e,h), (e,-h)\}$ to $\{(e+1,ih), (e+1,-ih)\}$.
It follows that $C^* = f(A^*)$. 

Now we look at $B^*$ and $D^*$, starting from $B'$ and $D'$.
Consider an element $a^eb^h \in B'$. Under $\psi^{-1}$, this element is mapped to $(e,h) \in B^*$.
From (\ref{aleft.eq}), the corresponding  element in $D'$ is $a^{e+1}b^{ih}$. Applying $\psi^{-1}$, we get the element 
$(e+1,ih) \in D^*$. Define  $f$ as before, i.e., $f(e,h) = (e+1,ih)$. Then 
it follows immediately that $D^* = f(B^*)$. 

In summary, $C^\ast=f(A^\ast)$ and $D^\ast=f(B^\ast)$, as desired.
\end{proof}

\begin{corollary}
\label{kk.cor}
Suppose $k$ is odd and $n = (k^2 + 1)/2$. Suppose that $(A^\prime,B^\prime)$  and $(C^\prime,D^\prime)$ are equivalent strongly symmetric $(k,k)$-near-factorizations of $D_{n}$ with $C^\prime=f_{i,j}(A^\prime)h$ and $D^\prime=h^{-1}f_{i,j}(B^\prime)$.  
Let $(A,B)$ and $(C,D)$ be obtained from $(A^\prime,B^\prime)$ and $(C^\prime,D^\prime)$ by applying the inverse P\^{e}cher transform.  Then $(A,B)$ and $(C,D)$ are equivalent symmetric near factorizations of $\zed_{2n}$.
%
%
\end{corollary}
\begin{proof}
This follows immediately from Theorem \ref{equiv2.thm} because 
$\GCD((k+1)/2, (k^2 + 1)/2)= 1$.
\end{proof}

\section{Nonequivalent near-factorizations in nonabelian groups}
\label{nonabelian.sec}

\subsection{$(a,a)$-near-factorizations in $D_{(a^2+1)/2}$}
\label{D41.sec}

In \cite{KPS}, $(a,a)$-near-factorizations in $\zed_{a^2+1}$ were enumerated for $a \leq 14$. It was shown in that paper that there are two $(9,9)$-near-factorizations in $\zed_{82}$, up to equivalence (there are no other   nonequivalent   $(a,a)$-near-factorizations in $\zed_{a^2+1}$  for $a$ odd, $a < 14$).
It is also known from Theorem \ref{translation.thm} that a near-factorization $(A,B)$ in a cyclic group has a ``translation'' $(A+g,B-g)$  that is symmetric. Hence there are two symmetric $(9,9)$-near-factorizations in $\zed_{82}$ that are nonequivalent. 
The two resulting strongly symmetric near-factorizations of $D_{41}$ obtained from the P\^{e}cher transform are nonequivalent, by Theorem \ref{equiv2.thm}. These two near-factorizations, which we denote by $(A_1,B_1)$ and $(A_2,B_2)$, are as follows:

\begin{align*}
A_1&=\{e,b^2,b^4,b^{37},b^{39},ab,ab^3,ab^{38},ab^{40}\}\\
B_1&=\{b^9,b^{14},b^{27},b^{32},a,ab^5,ab^{18},ab^{23},ab^{36}\} \\[8pt]
A_2&=\{e,b^8,b^{10},b^{31},b^{33},ab,ab^9,ab^{32},ab^{40}\}\\
B_2&=\{b^3,b^{14},b^{27},b^{38},a,ab^{11},ab^{17},ab^{24},ab^{30}\}.
\end{align*}

More generally, suppose that $a$ is any composite odd integer, say $a = jk$, where $j > 1$ and $k > 1$. Let $2n = a^2 + 1$; then $n$ is odd.  We can construct two SEDFs in $\zed_{2n}$ from the following two blowup sequences: $(a, a)$ and $(j,j,k,k)$
(blowup sequences are described in \cite{KPS,PS24}). 
The blowup sequence  $(a, a)$ yields the following SEDF:
\begin{align*} A_1 &= \{0, 1, \dots , a-1\} \\
B_1 &= \{a, 2a, \dots , a^2\}. 
\end{align*}
The blowup sequence  $(j,j,k,k)$ yields the following SEDF:
\begin{align*} A_2 &= \{ik^2 + h : 0 \leq i \leq j-1, 0 \leq h \leq k-1\} \\
B_2 &= \{(j-1+ji)k^2 + hk : 0 \leq i \leq j-1, 1 \leq h \leq k\}. 
\end{align*}

We show that these two SEDFs are non-equivalent as follows. The blowup sequence $(a, a)$ yields the SEDF $(A_1,B_1)$, where $A_1$ and $B_1$ both contain arithmetic progressions of length $a$. Any SEDF equivalent  to $(A_1,B_1)$ also has this property (provided that we let the arithmetic sequence ``wrap around''). The same property holds for the near-factorizations derived from this SEDF. On the other hand, it is easy to see that the SEDFs (and near-factorizations) derived from the  blowup sequence  $(j,j,k,k)$ do not contain arithmetic sequences of length $a$.

We can transform these two nonequivalent near-factorizations into nonequivalent symmetric near-factorizations of $\zed_{2n}$ using Theorem \ref{translation.thm}. Then we obtain two strongly symmetric near-factorizations of $D_n$ using the Pecher transform.  These near-factorizations are nonequivalent due to Corollary \ref{kk.cor}. Thus we have proven the following.

\begin{theorem}
Suppose $a$ is any composite odd integer, say $a = jk$, where $j > 1$ and $k > 1$.
Let $n = (a^2+1)/2$. Then there are nonequivalent near-factorizations of $D_n$.
\end{theorem}

\subsection{$(9,21)$-near-factorizations in $D_{95}$}
\label{D95.sec}

We have shown by computer that there are two (symmetric) $(9,21)$-near-factorizations in $\zed_{190}$, up to equivalence. These two near-factorizations, which we denote by $(A_1,B_1)$ and $(A_2,B_2)$, are obtained  by setting
$A_i=A_i' \cup (-A_i')$  and $B_i=B_i' \cup(-B_i')$, $i = 1,2$, where

\begin{align*}
 A_1'&=\{0, 1, 2, 3, 4\}\\
 B_1'&=\{5, 14, 23, 32, 41, 50, 59, 68, 77, 86, 95\}\\[8pt]
 A_2'&=\{0, 1, 8, 9, 10\}\\   
 B_2'&=\{11, 14, 17, 38, 41, 44, 65, 68, 71, 92, 95\}.
 \end{align*}

We obtain two
strongly symmetric near-factorizations of $D_{95}$  from the P\^{e}cher transform.
Theorem \ref{equiv2.thm} does not establish that these strongly symmetric near-factorizations of $D_{95}$ are nonequivalent, because the required $\GCD$ conditions are not satisfied: 
\[\GCD \left( \frac{9+1}{2}, 95 \right) = 5 > 1.\]
However, we tested the two resulting near-factorizations of $D_{95}$ using the computer and we found that they are nonequivalent.

\subsection{$(7,7)$-near-factorizations in $D_5 \times \zed_{5}$}

At the present time, there are almost no known examples of near-factorizations in nonabelian groups other than the dihedral groups. P\^{e}cher \cite{Pech} presented a $(7,7)$-near-factorization of $D_5 \times \zed_{5}$. We have enumerated all nonequivalent near-factorizations of this type and find that there are exactly two (of course, one of these is the near-factorization previously given in \cite{Pech}). 

In order to test equivalence, we need to determine
$\Aut(D_5 \times \zed_{5})$. It might not be immediately obvious, but it is not difficult to prove that
$\Aut(D_5 \times \zed_{5}) \cong \Aut(D_5) \times \Aut(\zed_{5})$. This can be shown as follows.
For convenience, we write the cyclic group of order five multiplicatively and denote it by $C_5 = \{1,c,c^2,c^3,c^4\}$.

$D_5 \times C_{5}$ has the presentation
\[
D_5 \times C_{5} = 
\left\langle
a,b,c : a^2=b^5= abab = c^5 = e, ac = ca, bc = cb
\right\rangle.
\]
Suppose $f \in \Aut(D_5 \times C_{5})$. Obviously $f$ is completely determined by
$x = f(a)$, $y=f(b)$ and $z = f(c)$. 
Because $Z(D_5 \times C_{5}) = \{e\} \times C_5$, it must be the case that  $z = c^k$, where $1 \leq k \leq 4$.
Because the only elements of order  two in $D_5 \times C_{5}$ are $(ab^j,1)$ (for $0 \leq j \leq 4$), 
it must be the case that $y = ab^j$ (where $0 \leq j \leq 4$). Finally, $x$ must have order five in $D_5 \times C_{5}$. Therefore $x = b^ic^k$, where  $i$ and $k$ are not both equal to $0$. Since $abab=e$, we must also have
$f(a)f(b)f(a)f(b) = xyxy = e$.
Hence,
\begin{align*}
e &= xyxy \\
&= b^ic^kab^jb^ic^kab^j\\
&= c^{2k} (b^i a) (b^{j+i}) (a b^j)\\
&= c^{2k} (ab^{-i})  b^{j+i} (a b^j)\\
&= c^{2k} (ab^{j}) (a b^j)\\
&= c^{2k}.
\end{align*}
Hence $k = 0$ and it follows that $f(b) = b^i$. Therefore $\Aut(D_5 \times \zed_{5}) \cong \Aut(D_5) \times \Aut(\zed_{5})$.

We have already discussed $\Aut(D_n)$ in Section \ref{notation.sec} and we have also observed that $\Aut(\zed_{n})$ just consists of  all mappings
$x \mapsto rx$, where $r \in \zed_{n}^*$.
Therefore an automorphism of $C_5$ is a mapping of the form 
$x \mapsto x^r$ where $r \in \{1,2,3,4\}$.

Our computation finds the following two nonequivalent near-factorizations of
$D_5 \times C_5$:
\begin{align*}
A_1&=\big\{(e,e),(e,c),(b,c^3),(b^{2},c^3),(a,e),(a,c),(ab,c^3)\big\}\\
B_1&=\big\{(b,c^2),(b^{4},c),(b^{4},c^3),(a,c^2),(ab^{2},c^2),(ab^{3},c),(ab^{3},c^3)\big\}\\
\intertext{and}
A_2&=\big\{(e,e),(e,c),(b,c^3),(b^{2},c^3),(a,e),(a,c),(ab^{4},c^3)\big\}\\
B_2&=\big\{(b,c^2),(b^{4},c),(b^{4},c^3),(a,c^2),(ab^{2},c),(ab^{2},c^3),(ab^{3},c^2)\big\}.
\end{align*}
P\^{e}cher \cite{Pech} provides this near-factorization of $D_5 \times C_5$:
\begin{align*}
A_{\text{P\^{e}cher}} &=\big\{
(e,e),
(a,e),
(e,c^3),
(a,c^3),
(ab,c^4),
(ab,c^4),
(b^2,c^4)
\big\}\\
B_{\text{P\^{e}cher}} &=\big\{
(a,c),
(b,c),
(ab^2,c),
(ab^3,c^3),
(b^4,c^3),
(ab^3,c^4),
(r^4,c^4)
\big\}.
\end{align*}

The mapping $\Phi_{g,e}$, where $g=(f_{0,0},x\mapsto x^2)$ is such that
\[
\Phi_{g,e}\big((A_{\text{P\^{e}cher}},B_{\text{P\^{e}cher}})\big) = (A_1,B_1).
\]
This shows that our solution $(A_1,B_1)$ and P\^{e}cher's solution are equivalent. 


\subsection{$(7,7)$-near-factorizations in $C_5^2\rtimes_{2}C_2$}

Let $G = C_5^2\rtimes_{2}C_2$ be the non-abelian group of order 50 having 
the following generators and relations presentation:
\[
G = \langle a,b,c | a^5=b^5=c^2=e, aba=b^{-1}, aca=c^{-1}, bc=cb \rangle
\]
The automorphism group $\Aut(G)$ has generators $x,y,z,w$, given by
\[
\begin{array}{@{}lll@{}}
w:a\mapsto a&
w:b\mapsto b^4c^4&
w:c\mapsto b\\
\\
x:a\mapsto a&
x:b\mapsto b^2&
x:c\mapsto c\\
\\
y:a\mapsto b&
y:b\mapsto b&
y:c\mapsto c\\
\\
z:a\mapsto ac&
z:b\mapsto b&
z:c\mapsto c.\\
\end{array}
\]
Our computation finds that there are two nonequivalent near-factorizations of
$G$:
\begin{align*}
A_1&=\big\{e,c,b,b^2c^2,a,ac,ab\big\}\\
B_1&=\big\{bc^4,b^4c,b^4c^4,ac^3,ab^2c^2,ab^3,ab^3c^3\big\}\\
\intertext{and}
A_2&=\big\{e,c,b,b^2c^2,a,ac,ab^4c\big\}\\
B_2&=\big\{bc^4,b^4c,b^4c^4,ac^3,ab^2c,ab^2c^3,ab^3c^4\big\}.
\end{align*}
P\^{e}cher provides this near factorization of $G$:
\begin{align*}
A_{\text{P\^{e}cher}} &=\big\{
e,a,b,c,ab^4,ac^4,b^2c^2
\big\}\\
B_{\text{P\^{e}cher}} &=\big\{
ab^2,ac^2,ab^3c^3,b^4c,bc^4,ab^2c^2,b^4c^4
\big\}.
\end{align*}

Observe that
\[
\Phi_{z,e}\big((A_{\text{P\^{e}cher}},B_{\text{P\^{e}cher}})\big) = (A_2,B_2),
\]
showing that our solution 2 and P\^{e}cher's solution are equivalent. 


\section{Summary and conclusion}

We showed that some known constructions of near-factorizations in dihedral groups yield equivalent near-factorizations. We also found some new examples and infinite classes of nonequivalent near-factorizations in nonabelian groups. Finally, we provided a detailed analysis of a construction for near-factorizations in dihedral groups from near-factorizations in cyclic groups, due to P\^{e}cher \cite{Pech}. One interesting open question is if Lemma \ref{hj.lem} and Theorem \ref{equiv2.thm} are true if the stated GCD conditions are not satisfied. 

We have focussed on near-factorizations in nonabelian groups in this paper. We mention that we have investigated 
near-factorizations of noncyclic abelian groups in another recent paper (\cite{KMS}). At the present time, there is no known example of a nontrivial near-factorization of any noncyclic abelian group.

\section*{Acknowledgements}
We thank Michael Epstein, Sophie Huczynska, and Shuxing Li for helpful discussions. 

\end{document}